\documentclass[10pt,a4paper]{article}
\usepackage[a4paper]{geometry}
\usepackage{amssymb,latexsym,amsmath,amsfonts,amsthm}
\usepackage{graphicx}

\usepackage{epsfig}

\newcommand{\supp}{\mathrm{supp}}

\newtheorem{theorem}{Theorem}[section]
\newtheorem{lemma}[theorem]{Lemma}

\theoremstyle{definition}
\newtheorem{definition}[theorem]{Definition}

\theoremstyle{remark}

\newtheorem{remark}[theorem]{Remark}

\numberwithin{equation}{section}

\title{On the convergence of type I Hermite-Pad\'e approximants for a class of meromorphic functions}

\date{\today}

\author{G. L\'{o}pez Lagomasino\footnotemark[2], S. Medina Peralta\footnotemark[2]}

\begin{document}

\maketitle
\renewcommand{\thefootnote}{\fnsymbol{footnote}}
\footnotetext[2]{Departamento de
Matem\'{a}ticas, Universidad Carlos III de Madrid, Avda. Universidad
30, 28911 Legan\'{e}s, Madrid, Spain.
Both authors were partially supported
by research grant MTM2012-36372-C03-01 of Ministerio de Econom\'{\i}a y Competitividad, Spain.}

\begin{abstract}
We study the convergence of type I Hermite-Pad\'e approximation for a class of meromorphic functions obtained by adding a vector of rational functions with real coefficients to a Nikishin system of functions.
\end{abstract}



\maketitle

\section{Introduction}
\label{section:intro}

Let $s$ be a finite  Borel measure with constant (not neessarily positive) sign whose support $\mbox{supp}(s)$ contains infinitely many points and is contained in the real line $\mathbb{R}$. If $\mbox{supp}(s)$ is an unbounded set we assume additionally that   $x^n \in L_1(s), n \in \mathbb{N}$.  By $\Delta = \mbox{Co}(\mbox{supp}(s))$ we denote the smallest interval which contains  $\supp(s)$.   We denote this class of measures by ${\mathcal{M}}(\Delta)$. Let
\[ \widehat{s}(z) = \int\frac{d s(x)}{z-x}
\]
be the Cauchy transform of $s$.

Given any positive integer $n \in \mathbb{N}$ there exist polynomials $Q_n,P_n$
satisfying:
\begin{itemize}
\item $\deg Q_n \leq n,\quad  \deg P_n \leq n-1, \quad Q_n \not\equiv 0,$
\item $(Q_n \widehat{s} -P_n)(z) = \mathcal{O}(1/z^{n+1}),\quad  z \to \infty.$
\end{itemize}
The ratio  $\pi_n = P_n/Q_n$ of any two such polynomials defines a unique rational function  called the $n$th term of the diagonal sequence of Pad\'e approximants to $\widehat{s}$. Cauchy transforms of measures are important: for example, many elementary functions may be expressed through them, the resolvent function of a bounded selfadjoint operator adopts that form, and they characterize all functions holomorphic in the upper half plane whose image lies in the lower half plane and can be extended continuously to the complement of a finite segment $[a,b]$ of the real line taking negative values for $z < a$ and positive values for $z > b$ (then $\mbox{supp}(s) \subset [a,b]$), see \cite[Theorem A.6]{KN}. Providing efficient methods for their approximation is a central question in the theory of rational approximation.

When $\Delta$ is bounded,  A.A. Markov proved in \cite{Mar} (in the context of the theory of continued fractions) that
\[\lim_{n \to \infty} \pi_n(z) = \widehat{s}(z)
\]
uniformly on each compact subset of $\overline{\mathbb{C}} \setminus \Delta$.  It is easy to deduce that the limit takes place with geometric rate. In the same year, see \cite{Sti}, T.J. Stieltjes obtained an analogous theorem for Cauchy transforms of measures with unbounded support contained in a half line, under the assumption that the moment problem for the sequence $\left(c_n\right)_{n\geq 0}, c_n = \int x^n ds(x),$ is determinate. It is well known that the moment problem for measures of bounded support  is always determinate; therefore, Stieltjes' theorem contains Markov's result. In \cite{Car}, T. Carleman  proved  when $\Delta \subset \mathbb{R}_+$ that
\begin{equation} \label{Carle} \sum_{n \geq 1} |c_{n}|^{-1/2n} = \infty
\end{equation}
is sufficient for the moment problem to be determinate. For an arbitrary measure  $s \in \mathcal{M}(\Delta)$, where $\Delta$ is contained in a half line, we say that it satisfies Carleman's condition if after an affine transformation which takes $\Delta$ into $\mathbb{R}_+$ the image measure satisfies Carleman's condition.

In an attempt to extend Markov's theorem to a general class of meromorphic functions, A.A. Gonchar considered functions of the form $\widehat{s} + r$ where $r$ is a rational function whose poles lie in $\mathbb{C} \setminus \Delta$. In \cite{Gon0}, he proved that if  $\Delta$ is a bounded interval and $s' > 0$ a.e. on $\Delta$, then
\begin{equation} \label{MS} \lim_{n\to \infty} \frac{P_n}{Q_n}(z) = \widehat{s}(z)+r(z),
\end{equation}
uniformly on each compact subset of (inside) $\mathbb{C} \setminus \Delta$. Here, $\{P_n/Q_n\}_{n\geq 0}$ denotes the diagonal sequence of Pad\'e approximants of $\widehat{s}+r$, showing, additionally, that each pole of $r$ in $\mathbb{C} \setminus \Delta$ ``attracts'' as many zeros of $Q_n$ as its order and the remaining zeros of $Q_n$ accumulate on $\Delta$ as $n \to \infty$. Later, in \cite{Rak1} E.A. Rakhmanov obtained a full extension of Markov's theorem when $r$ has real coefficients and proved that if $r$ has complex coefficients then such a result is not possible  without extra assumptions on the measure $s$. The case of unbounded $\Delta$ was solved in \cite{L1}, when $r$ has real coefficients, and \cite{L2}, when $r$ has complex coefficients.

Pad\'e  approximation has two natural extensions to the case of vector rational approximation.  These extensions  were introduced by Hermite in order to study the transcendency of $e$. Other applications in number theory have been obtained. See \cite{Ass} for a survey of results in this direction.

Given a system of finite Borel measures $S = (s_1,\ldots,s_m)$ with constant sign and a multi-index ${\bf n} = (n_1,\ldots,n_m) \in \mathbb{Z}_+^m \setminus \{{\bf 0}\}, |{\bf n}|= n_1+\cdots+n_m$, where $\mathbb{Z}_+$ denotes the set of non-negative integers and $\bf 0$ the $m$-dimensional zero vector, their exist polynomials $a_{{\bf n},j}, j=0,\ldots,m,$ not all identically equal to zero, such that:
\begin{itemize}
\item[\textit{i)}] $\deg a_{{\bf n},j} \leq n_j -1, j=1,\ldots,m, \quad \deg a_{{\bf n},0} \leq \max(n_j) -2,$ $\,\,\,(\deg a_{{\bf n},j} \leq -1$ means that $a_{{\bf n},j} \equiv 0$)
\item[\textit{ii)}] $a_{{\bf n},0}(z) + \sum_{j=1}^m a_{{\bf n},j}(z) \widehat{s}_j(z) = \mathcal{O}(1/z^{|{\bf n}|}),\,\,\, z \to \infty$.
\end{itemize}
Analogously, there exist polynomials $Q_{\bf n}, P_{{\bf n},j}, j=1,\ldots,m$, satisfying:
\begin{itemize}
\item[\textit{i')}] $\deg Q_{\bf n} \leq |{\bf n}|, Q_{\bf n} \not\equiv 0, \quad \deg P_{{\bf n},j} \leq |{\bf n}|-1, j=1,\ldots,m,$
\item[\textit{ii')}] $Q_{\bf n}(z) \widehat{s}_j (z) - P_{{\bf n},j}(z) = \mathcal{O}(1/z^{n_j +1}), \quad z\to \infty, \quad j=1,\ldots,m.$
\end{itemize}
The existence of the vector of polynomials $(a_{{\bf n},1}, \ldots, a_{{\bf n},m})$ reduces to solving a homogeneus linear system of $|{\bf n}|-1$ equations on the total number of $|{\bf n}|$ coefficients  of $(a_{{\bf n},1}, \ldots, a_{{\bf n},m})$, and the existence of $Q_{\bf n}$ reduces to solving a homogeneus linear system of $|{\bf n}|$ equations on the total number of $|{\bf n}|+1$ coefficients of the  polynomial $Q_{\bf n}$; therefore, a non-trivial solution is guaranteed.
The polynomials $a_{{\bf n},0}$ and $P_{{\bf n},j},j=1,\ldots,m,$ are uniquely determined from \textit{ii)} and \textit{ii')} once their partners are found.\par
Traditionally, the systems of polynomials $(a_{{\bf n},0}, \ldots,a_{{\bf n},m})$ and $(Q_{\bf n},P_{{\bf n},1},\ldots,P_{{\bf n},m})$ have been called type I and type II Hermite-Pad\'e approximants of $(\widehat{s}_1,\ldots,\widehat{s}_m)$, respectively. When $m=1$ both definitions reduce to that of classical Pad\'e approximation.\par

Before stating our main result, let us introduce what is called a Nikishin system of measures to which we will restrict our study.
Let $\Delta_{\alpha}, \Delta_{\beta}$ be two intervals contained in the real line which have at most one point in common, $\sigma_{\alpha} \in {\mathcal{M}}(\Delta_{\alpha}), \sigma_{\beta} \in {\mathcal{M}}(\Delta_{\beta})$, and $\widehat{\sigma}_{\beta} \in L_1(\sigma_{\alpha})$. With these two measures we define a third one as follows (using the differential notation)
\[ d \langle \sigma_{\alpha},\sigma_{\beta} \rangle (x) := \widehat{\sigma}_{\beta}(x) d\sigma_{\alpha}(x).
\]
Above, $\widehat{\sigma}_{\beta}$ denotes the Cauchy transform of the measure $\sigma_{\beta}$.  The more appropriate notation $\widehat{\sigma_{\beta}}$ causes space consumption and aesthetic inconveniences. We need to take consecutive products of measures; for example,
\[\langle \sigma_{\gamma},  \sigma_{\alpha},\sigma_{\beta} \rangle :=\langle \sigma_{\gamma}, \langle \sigma_{\alpha},\sigma_{\beta} \rangle \rangle. \]
 Here, we assume not only that $\widehat{\sigma}_{\beta} \in L_1(\sigma_{\alpha})$ but also $\langle \sigma_{\alpha},\sigma_{\beta} \widehat{\rangle} \in L_1(\sigma_{\gamma})$ where $\langle \sigma_{\alpha},\sigma_{\beta} \widehat{\rangle}$ denotes the Cauchy transform of $\langle \sigma_{\alpha},\sigma_{\beta}  {\rangle}$. Inductively, one defines products of a finite number of measures.

\begin{definition} \label{Nikishin} Take a collection  $\Delta_j, j=1,\ldots,m,$ of intervals such that, for   $j=1,\ldots,m-1$
\[ \Delta_j \cap \Delta_{j+1} = \emptyset, \qquad \mbox{or} \qquad \Delta_j \cap \Delta_{j+1} = \{x_{j,j+1}\},
\]
where $x_{j,j+1}$ is a single point. Let $(\sigma_1,\ldots,\sigma_m)$ be a system of measures such that $\mbox{Co}(\supp (\sigma_j)) = \Delta_j, \sigma_j \in {\mathcal{M}}(\Delta_j), j=1,\ldots,m,$  and
\begin{equation} \label{eq:autom}
\langle \sigma_{j},\ldots,\sigma_k  {\rangle} := \langle \sigma_j,\langle \sigma_{j+1},\ldots,\sigma_k\rangle\rangle\in {\mathcal{M}}(\Delta_j),  \qquad  1 \leq j < k\leq m.
\end{equation}
When $\Delta_j \cap \Delta_{j+1} = \{x_{j,j+1}\}$ we also assume that $x_{j,j+1}$ is not a mass point of either $\sigma_j$ or $\sigma_{j+1}$.
We say that $(s_{1,1},\ldots,s_{1,m}) = {\mathcal{N}}(\sigma_1,\ldots,\sigma_m)$, where
\[ s_{1,1} = \sigma_1, \quad s_{1,2} = \langle \sigma_1,\sigma_2 \rangle, \ldots \quad , s_{1,m} = \langle \sigma_1, \sigma_2,\ldots,\sigma_m  \rangle
\]
is the Nikishin system of measures generated by $(\sigma_1,\ldots,\sigma_m)$.
\end{definition}

It is not difficult to show (see \cite[Theorem 1.5]{LM}) that if $\sigma_1 = s_{1,1}$ satisfies Carleman's condition \eqref{Carle} then $s_{1,k}, k=2,\ldots,m,$ also satisfies that condition.

Initially, E.M. Nikishin in \cite{Nik}  restricted himself to measures with bounded support and no intersection points between consecutive $\Delta_j$.
Definition \ref{Nikishin} includes interesting examples  which appear in practice (see, \cite[Subsection 1.4]{FL4II}). We follow the approach of \cite[Definition 1.2]{FL4II}  assuming additionally the existence of all the moments of the generating measures. This is done only for the purpose of simplifying the presentation without affecting too much the generality. However, we wish to point out that the results of this paper have appropriate formulations with the definition given in \cite{FL4II} of a Nikishin system.

In \cite[Lemma 2.9]{FL4II} it was shown that if $(\sigma_1,\ldots,\sigma_m)$ is a generator of a Nikishin system then $(\sigma_m,\ldots,\sigma_1)$ is also a generator (as well as any subsystem of consecutive measures drawn from them). When the supports are bounded and consecutive supports do not intersect this is trivially true. In the following, for $1\leq j\leq k\leq m$ we denote
\[ s_{j,k} := \langle \sigma_j,\sigma_{j+1},\ldots,\sigma_k \rangle, \qquad s_{k,j} := \langle \sigma_k,\sigma_{k-1},\ldots,\sigma_j \rangle.
\]

From the definition, type II Hermite-Pad\'e approximation is easy to view as an approximating scheme of the vector function
$(\widehat{s}_{1,1},\ldots,\widehat{s}_{1,m})$ by considering a sequence of vector rational functions of the form $(P_{{\bf n},1}/Q_{\bf n},\ldots,P_{{\bf n},m}/Q_{\bf n}), {\bf n} \in \Lambda \subset \mathbb{Z}_+^m$, where $Q_{\bf n}$ is a common denominator for all components, in \cite{BusLop} the authors obtain an analogue of Markov's Theroem for type II Hermite-Pad\'e approximation  with respect to a Nikishin system. For type I Hermite-Pad\'e approximation is not obvious   what is the object to be approximated or even what should be considered as the approximant. This problem was solved in \cite{LM} when the system of measures $S = (s_{1,1},\ldots,s_{1,m})$ is a Nikshin system. Later, in \cite{FLM} the authors consider a  type II Hermite-Pad\'e approximation with respect to a system of meromorphic functions  of the form  ${\bf f}=(f_{1},\ldots, f_{m}) = \widehat{\bf s} + {\bf r}$, where
\begin{equation}\label{fs}
f_{j}(z)=\widehat{s}_{1,j}(z)+ r_{j}(z), \qquad j=1,\ldots,m.
\end{equation}
Here,  ${\bf r}=(r_1,\ldots,r_m)=\left(\displaystyle \frac{v_{1}}{t_{1}},\ldots, \frac{v_{m}}{t_{m}}\right)$, is a vector of rational fracctions with real coefficients  such that $\deg v_{j}<\deg t_{j}=d_{j}$  for every $j=1,\ldots,m$. We assume that $v_j/t_j, j=1,\ldots,m$ is irreducible and ${\bf s}=(s_{1,1},\ldots,s_{1,m})$ is a Nikshin system. For  type II Hermite-Pad\'e approximation with respect to the system  ${\bf f}$ also we have an entension of Markov's theorem.\par
Our goal is to obtain a similar result for type I Hermite-Pad\'e approximation respect with to the system  ${\bf f}$.

\begin{theorem} \label{th1}  Let  $\Lambda \subset \mathbb{Z}_+^{m}$ be an infinite sequence of distinct muti-indices. Consider the corresponding sequence $\left(a_{{\bf n},0},\ldots, a_{{\bf n},m}\right), {\bf n} \in \Lambda,$ of  type I  Hermite-Pad\'e approximants of ${\bf f}$.
Assume  that the rational functions $r_{1},\ldots r_{m}$  have real coefficients and their  poles  lie in $\mathbb{C}\setminus (\Delta_{1}\cup \Delta_{m})$, for $j\neq k$ the poles of $r_j$ and $r_k$ are distinct.
Assume that
\begin{equation}\label{cond1} \sup_{{\bf n}\in \Lambda}\left(\max_{j=1,\ldots,m}(n_j) - \min_{k=1,\ldots,m}(n_k) \right)\leq C < \infty, \end{equation}
and that either $\Delta_{m-1}$ is bounded away from $\Delta_m$  or $\sigma_m$ satisfies Carleman's condition.
Then,
\begin{equation} \label{fund1} \lim_{{\bf n} \in \Lambda}  \frac{a_{{\bf n}, j}}{a_{{\bf n},m}} =  (-1)^{m-j}\widehat{s}_{m,j+1}, \qquad j=1,\ldots,m-1,
\end{equation}
and
\begin{equation} \label{fund0} \lim_{{\bf n} \in \Lambda}  \frac{a_{{\bf n}, 0}}{a_{{\bf n},m}} = (-1)^{m}\widehat{s}_{m,1}-\sum_{j=1}^{m-1}(-1)^{m-j}r_j\widehat{s}_{m,j+1}+r_m.
\end{equation}
uniformly on each compact subset $K$ contained in $(\mathbb{C}\setminus \Delta_{m})^{\prime},$ the set obtained deleting from $\mathbb{C}\setminus \Delta_{m} $ the poles of all the $r_j$.
 Additionally, let $\zeta$ be a zero of $T=t_1\cdot t_2\ldots t_m$ of multiplicity $\kappa$,  then
   for each $\varepsilon >0$ sufficiently small there exists an $N$ such that for all ${\bf n} \in \Lambda, |{\bf n}| > N$ and all $j=1,\ldots,m$  $a_{{\bf n},j}$ has exactly   $\kappa$ zeros in $\{z:|z-\zeta| < \varepsilon\}$ and the rest of their zeros acumalate on $\Delta_m\cup \{\infty\}$.
\end{theorem}

Notice that the rational fractions $(r_1,\ldots,r_m)$ do not play any role in the expression of the limit  of $(\frac{a_{{\bf n}, 1}}{a_{{\bf n},m}},\ldots,\frac{a_{{\bf n}, m-1}}{a_{{\bf n},m}})$. On the other hand, all the information of $(r_1,\ldots,r_m)$ is contained in the expression of the limit of $\frac{a_{{\bf n}, 0}}{a_{{\bf n},m}}$.\par

This paper is organized as follows. Section 2 contains some auxiliary results. In Section 3 we prove Theorem \ref{th1} and describe some other consequences of the main result and further extensions.

\section{Some auxiliary results} \label{aux}

We begin with a result, which appears in \cite[Lemma 2.1]{LM} and is easy to deduce, which  gives an integral representation of the remainder of type I multi-point  Hermite-Pad\'e approximants.

\begin{lemma} \label{reduc} Let $(s_{1,1},\ldots,s_{1,m}) = \mathcal{N}(\sigma_1,\ldots,\sigma_m)$ be given. Assume that there exist polynomials with real coefficients $a_0,\ldots,a_m$ and a polynomial $w$ with real coefficients whose zeros lie in $\mathbb{C} \setminus \Delta_1$  such that
\[\frac{\mathcal{A}(z)}{w(z)} \in \mathcal{H}(\mathbb{C} \setminus \Delta_1)\qquad \mbox{and} \qquad \frac{\mathcal{A}(z)}{w(z)} = \mathcal{O}\left(\frac{1}{z^N}\right), \quad z \to \infty,
\]
where $\mathcal{A}  := a_0 + \sum_{k=1}^m a_k  \widehat{s}_{1,k} $ and $N \geq 1$. Let $\mathcal{A}_1  := a_1 + \sum_{k=2}^m a_k  \widehat{s}_{2,k} $. Then
\begin{equation} \label{eq:3}
\frac{\mathcal{A}(z)}{w(z)} = \int \frac{\mathcal{A}_1(x)}{(z-x)} \frac{d\sigma_1(x)}{w(x)}.
\end{equation}
If $N \geq 2$, we also have
\begin{equation} \label{eq:4}
\int x^{\nu}  \mathcal{A}_1(x)  \frac{d\sigma_1(x)}{w(x)}, \qquad \nu = 0,\ldots, N -2.
\end{equation}
In particular, $\mathcal{A}_1$ has at least $N -1$ sign changes in  $\stackrel{\circ}{\Delta}_1 $ (the interior on $\Delta_1$ in $\mathbb{R}$ with the usual topology).
\end{lemma}

Some relations concerning the reciprocal and ratio of Cauchy transforms of measures will be useful.
It is  known that for each $\sigma
\in {\mathcal{M}}(\Delta),$ where $\Delta$ is contained in a half line, there exists a measure $\tau \in
{\mathcal{M}}(\Delta)$ and ${\ell}(z)=a z+b, a = 1/|\sigma|, b \in {\mathbb{R}},$ such that
\begin{equation} \label{s22}
{1}/{\widehat{\sigma}(z)}={\ell}(z)+ \widehat{\tau}(z),
\end{equation}
where $|\sigma|$ is the total variation of the measure $\sigma.$  See  \cite[Appendix]{KN} and \cite[Theorem 6.3.5]{stto} for measures with compact support, and \cite[Lemma 2.3]{FL4} when the support is contained in a half line. If $\sigma$ satisfies Carleman's condition then $\tau$ satisfies that condition (see \cite[Theorem 1.5]{LM}).

We call $\tau$  the inverse measure of $\sigma.$ They appear frequently in our reasonings, so we will fix a
notation to distinguish them. In relation with measures  denoted with $s$ they will carry over to them the
corresponding sub-indices. The same goes for the  polynomials
$\ell$. For example,
\[
{1}/{\widehat{s}_{j,k}(z)}  ={\ell}_{j,k}(z)+
\widehat{\tau}_{j,k}(z).
\]
We also write
\[
{1}/{\widehat{\sigma}_{\alpha}(z)} ={\ell}_{\alpha }(z)+
\widehat{\tau}_{\alpha }(z).
\]
Sometimes we write $\langle
\sigma_{\alpha},\sigma_{\beta} \widehat{\rangle}$ in place of $\widehat{s}_{\alpha,\beta}$.  In \cite[Lemma 2.10]{FL4}, several formulas involving ratios of Cauchy transforms were proved. The most useful ones  in this paper establish that
\begin{equation} \label{4.4}
\frac{\widehat{s}_{1,k}}{\widehat{s}_{1,1}} =
\frac{|s_{1,k}|}{|s_{1,1}|} - \langle \tau_{1,1},\langle s_{2,k},\sigma_1
\rangle \widehat{\rangle}  , \qquad  1=j < k \leq m.
\end{equation}

The notion of convergence in Hausdorff content plays a central  role in the proof of Theorem~\ref{th1}. Let $B$ be a subset of the complex plane $\mathbb{C}$. By
$\mathcal{U}(B)$ we denote the class of all coverings of $B$ by at
most a numerable set of disks. Set
$$
h(B)=\inf\left\{\sum_{i=1}^\infty
|U_i|\,:\,\{U_i\}\in\mathcal{U}(B)\right\},
$$
where $|U_i|$ stands for the radius of the disk $U_i$. The quantity
$h(B)$ is called the $1$-dimensional Hausdorff content of the
set $B$.

Let $(\varphi_n)_{n\in\mathbb{N}}$ be a sequence of complex functions
defined on a domain $D\subset\mathbb{C}$ and $\varphi$ another
function defined on $D$ (the value $\infty$ is permitted). We say that
$(\varphi_n)_{n\in\mathbb{N}}$ converges in Hausdorff content to
the function $\varphi$ inside $D$ if for each compact
subset $\mathcal{K}$ of $D$ and for each $\varepsilon
>0$, we have
$$
\lim_{n\to\infty} h\{z\in K :
|\varphi_n(z)-\varphi(z)|>\varepsilon\}=0
$$
(by convention $\infty \pm \infty = \infty$). We denote this writing $h$-$\lim_{n\to \infty} \varphi_n =
\varphi$ inside $D$.

To obtain Theorem \ref{th1} we first prove \eqref{fund1} and \eqref{fund0} with convergence in Hausdorff content in place of uniform convergence (see Lemma \ref{CCTI} below). We need the following notion.\par

Let $s \in \mathcal{M}(\Delta)$ where $\Delta$ is contained in a half line of the real axis. Fix an arbitrary $\kappa \geq -1$. Consider a sequence of polynomials $(w_n)_{n \in \Lambda},  \Lambda \subset \mathbb{Z}_+,$ such that $\deg w_n = \kappa_n \leq 2n + \kappa +1$, whose zeros lie in $\mathbb{R} \setminus \Delta$. Let $(R_n)_{n \in \Lambda}$ be a sequence of rational functions $R_n = p_n/q_n$ with real coefficients satisfying the following conditions for each $n \in \Lambda$:
\begin{itemize}
\item[a)] $\deg p_n \leq n + \kappa,\quad  \deg q_n \leq n, \quad q_n \not\equiv 0,$
\item[b)] $ {(q_n \widehat{s} - p_n)(z)}/{w_n} = \mathcal{O}\left( {1}/{z^{n+1 - \ell}}\right) \in \mathcal{H}(\mathbb{C}\setminus \Delta), z \to \infty,$ where $\ell \in \mathbb{Z}_+$ is fixed.
\end{itemize}
We say that $(R_n)_{n \in \Lambda}$ is a sequence of incomplete diagonal multi-point Pad\'e approximants of $\widehat{s}$.\par
Notice that in this construction for each $n \in \Lambda$ the number of free parameters equals $2n + \kappa +2$ whereas the number of homogeneous linear equations to be solved in order to find $q_n$ and $p_n$ is equal to $2n + \kappa - \ell + 1$. When $\ell =0$ there is only one more parameter than equations and  $R_n$ is defined uniquely coinciding with a (near) diagonal multi-point Pad\'e approximation. When $\ell \geq 1$ uniqueness is not guaranteed, thus the term incomplete.\par
For sequences of incomplete diagonal multi-point Pad\'e approximants, the following Stieltjes type theorem was proved in \cite[Lemma 2]{BusLop} in terms of convergence in logaritmic capacity and we reformulate it using $1$-Hausdorff content. The proof is basically the same.

\begin{lemma} \label{BusLop}
Let $s \in \mathcal{M}(\Delta)$ be given where $\Delta$ is contained in a half line. Assume that $(R_n)_{n \in \Lambda}$ satisfies {\rm a)-b)} and either the number of zeros of $w_n$ lying on a bounded segment of $\mathbb{R} \setminus \Delta$ tends to infinity as $n\to\infty, n \in \Lambda$, or $s$ satisfies Carleman's condition.
Then
\[ h-\lim_{n \in \Lambda} R_n = \widehat{s},\qquad \mbox{inside}\qquad  \mathbb{C} \setminus \Delta.
\]
\end{lemma}

Let $(s_{1,1},\ldots,s_{1,m}) = \mathcal{N}(\sigma_1,\ldots,\sigma_m)$, and ${\bf n} \in \mathbb{Z}_+^{m}  \setminus \{\bf 0\}$ be given.  Fix   $M \in \mathbb{Z}_+$. Consider a vector polynomial $\left(p_{{\bf n},0},\ldots, p_{{\bf n},m}\right),$ not  identically equal zero, which satisfies:
\begin{itemize}
\item[a')] $\deg p_{{\bf n},j} \leq n_j -1, j=1,\ldots,m,$
\item[b')] $p_{{\bf n},0} + \sum_{j= 1}^m p_{{\bf n},j} \widehat{s}_{1,j} = \mathcal{O}(1/z^{|{\bf n}| -M}) \in \mathcal{H}(\mathbb{C} \setminus \Delta_1)$
\end{itemize}
We call  $\left(p_{{\bf n},0},\ldots, p_{{\bf n},m}\right)$ an incomplete type I  Hermite-Pad\'e  approximation of $(s_{1,1},\ldots,s_{1,m})$ with respect to  $\bf n$.

The following lemma  is an extended version of Lemma~\ref{BusLop} and is   contained in \cite[Lemma 3.1]{LM}   for the case  when $M = 0$.

\begin{lemma} \label{CCTI}  Let ${\bf s}= (s_{1,1},\ldots,s_{1,m}) = \mathcal{N}(\sigma_1,\ldots,\sigma_m)$ and $\Lambda \subset \mathbb{Z}_+^{m}$ be an infinite sequence of distinct multi-indices. Fix  $M \in \mathbb{Z}_+$. Consider a sequence of incomplete type I multi-point Hermite-Pad\'e approximants of ${\bf s}$ with respect to ${{\bf n} \in \Lambda}$.
Assume that $(\ref{cond1})$ takes place and that either $\Delta_{m-1}$ is bounded away from $\Delta_m$ or $\sigma_m$ satisfies $(\ref{Carle})$. Then,
for each fixed $j=0,\ldots,m-1$
\begin{equation} \label{convHaus}
h-\lim_{{\bf n}\in \Lambda}\frac{p_{{\bf n}, j}}{p_{{\bf n},m}} = (-1)^{m-j}\widehat{s}_{m,j+1}, \quad h-\lim_{n\in \Lambda}\frac{p_{{\bf n}, m}}{p_{{\bf n},j}} = \frac{(-1)^{m-j}}{\widehat{s}_{m,j+1} },
\end{equation}
inside $\mathbb{C} \setminus \Delta_m$.
\end{lemma}

\begin{proof}
If $m=1$ the statement reduces directly to Lemma \ref{BusLop}, so without loss of generality we can assume that $m \geq 2$. Fix ${\bf n} \in \Lambda$.
 Denote
 $$ \mathcal{A}_{{\bf n},j}(z) := p_{{\bf n},j}(z) + \sum_{k= j+1}^m p_{{\bf n},k}(z) \widehat{s}_{j+1,k}(z), \qquad j=0,\ldots,m-1.$$

From  Lemma \ref{reduc} it follows that
 $\mathcal{A}_{{\bf n},1}$ has at least  $|{\bf n}| -M-1$ simple zeros in the interior of  $\Delta_1$. Therefore, there exists a polynomial
$w_{{\bf n},1}, \deg w_{{\bf n},1} = |{\bf n}| -M-1,$ whose zeros lie on $\Delta_1$ such that
\begin{equation} \label{A1}
\frac{\mathcal{A}_{{\bf n},1}}{w_{{\bf n},1}}  \in \mathcal{H}(\mathbb{C} \setminus \Delta_2).
\end{equation}
Set $\overline{n}_j = \max \{n_k: k=j,\ldots,m\}$. Taking into account the degrees of the polynomials $p_{{\bf n},j}$ and $w_{{\bf n},1}$ it follows that
\begin{equation} \label{O1}
 \frac{\mathcal{A}_{{\bf n},1}}{w_{{\bf n},1}} = \mathcal{O}\left(\frac{1}{z^{|{\bf n}| -M- \overline{n}_1}}\right), \qquad z\to \infty.
\end{equation}

From  (\ref{A1}), (\ref{O1}), and Lemma \ref{reduc} we have
that $\mathcal{A}_{{\bf n},2}$ has at least $|{\bf n}| - M-\overline{n}_1 -1$ sign changes in  $\stackrel{\circ}{\Delta}_2$.  Therefore, there exists a polynomial  $w_{{\bf n},2}, \deg w_{{\bf n},2} = |{\bf n}| - M-\overline{n}_1 -1$, whose zeros lie on $\Delta_2$, such that
\[ \frac{\mathcal{A}_{{\bf n},2}}{w_{{\bf n},2}}  \in \mathcal{H}(\mathbb{C} \setminus \Delta_3), \quad \mbox{and} \quad \frac{\mathcal{A}_{{\bf n},2}}{w_{{\bf n},2}} = \mathcal{O}\left(\frac{1}{z^{|{\bf n}| -M- \overline{n}_1 - \overline{n}_2}}\right) , \quad z \to \infty.
\]

Iterating this process, using Lemma \ref{reduc} several times, on step $j,\, j \in \{1,\ldots,m\},$ we find that there exists a polynomial
$w_{{\bf n},j}, \deg w_{{\bf n},j} = |{\bf n}| -M-\overline{n}_1 -\cdots - \overline{n}_{j-1}- 1,$ whose zeros are points where $\mathcal{A}_{{\bf n},j}$ changes sign on $\Delta_j$ such that
\begin{equation} \label{Anj}
\frac{\mathcal{A}_{{\bf n},j}}{w_{{\bf n},j}}  \in \mathcal{H}(\mathbb{C} \setminus \Delta_{j+1}), \quad \mbox{and} \quad \frac{\mathcal{A}_{{\bf n},j}}{w_{{\bf n},j}} = \mathcal{O}\left(\frac{1}{z^{|{\bf n}| -M- \overline{n}_1 - \cdots-\overline{n}_j}}\right) , \quad z \to \infty.
\end{equation}
This process concludes as soon as $|{\bf n}| - M-\overline{n}_1 - \cdots-\overline{n}_j \leq 0$. Since $\lim_{{\bf n} \in \Lambda} |{\bf n}| = \infty$, because of (\ref{cond1}) we can always take $m$ steps for all ${\bf n} \in \Lambda$ with $|{\bf n}|$ sufficiently large. In what follows, we only consider such ${\bf n}$'s.

When $n_1 = \overline{n}_1\geq \cdots \geq n_m = \overline{n}_m$, we obtain that $\mathcal{A}_{{\bf n},m} \equiv p_{{\bf n},m}$ has at least  $n_m -M-1$ sign changes on $\Delta_m$. If $M=0$  since $\deg p_{{\bf n},m} \leq n_m-1$ this means that $\deg p_{{\bf n},m} = n_m-1$ and all its zeros lie on $\Delta_m$. (In fact, in this case we can prove that $\mathcal{A}_{{\bf n},j}, j=1,\ldots,m$ has exactly $|{\bf n}| - n_1-\cdots -n_{j-1}$ zeros in $\mathbb{C} \setminus \Delta_{j+1}$ that they are all simple and lie in the interior of $\Delta_j$, where $\Delta_{m+1} = \emptyset$).

In general,  $p_{{\bf n},m}$ has at least $|{\bf n}|-M - \overline{n}_1 - \cdots-\overline{n}_{m-1}-1$ sign changes on $\Delta_m$; therefore, the number of zeros of $p_{{\bf n},m}$ which may lie outside of $\Delta_m$ is bounded by
\[ \deg p_{{\bf n},m} - (|{\bf n}|-M - \overline{n}_1 - \cdots-\overline{n}_{m-1}-1) \leq \sum_{k=1}^{m-1} \overline{n}_k -n_k \leq (m-1)C+M,\]
where $C$ is the constant given in (\ref{cond1}), which does not depend on ${\bf n} \in \Lambda$.

For $j=m-1$   there exists $w_{{\bf n},m-1}, \deg w_{{\bf n},m-1} = |{\bf n}| -M- \overline{n}_{1} - \cdots - \overline{n}_{m-2} -1,$ whose zeros lie on $\Delta_{m-1}$ such that
$$\frac{\mathcal{A}_{{\bf n},m-1}}{w_{{\bf n},m-1}} = \frac{p_{{\bf n},m-1} + p_{{\bf n},m} \widehat{\sigma}_m}{w_{{\bf n},m-1}} \in \mathcal{H}(\mathbb{C} \setminus \Delta_m),$$
and
$$  \frac{\mathcal{A}_{{\bf n},m-1}}{w_{{\bf n},m-1}} = \mathcal{O}\left(\frac{1}{z^{|{\bf n}| -M- \overline{n}_1 - \cdots-\overline{n}_{m-1}}}\right) ,\qquad  z \to \infty,
$$
where $\deg  p_{{\bf n},m-1} \leq n_{m-1} -1, \deg  p_{{\bf n},m} \leq n_{m} -1$. Thus, using (\ref{cond1}) it is easy to check that  $(p_{{\bf n},m-1}/p_{{\bf n},m})_{ n \in \Lambda}$ forms a sequence of  incomplete diagonal multi-point Pad\'e approximants of $-\widehat{\sigma}_m$ satisfying a)-b) with appropriate values of $n,\kappa$ and $\ell$. Since $\sigma_m$ satisfies Carleman's condition, due  to Lemma \ref{BusLop} it follows that
\[ h-\lim_{n\in \Lambda} \frac{p_{{\bf n},m-1}}{p_{{\bf n},m}} = - \widehat{\sigma}_m, \qquad \mbox{inside} \qquad \mathbb{C} \setminus \Delta_m.
\]
Dividing by $\widehat{\sigma}_m$ and using (\ref{s22}), we also have
\[ \frac{\mathcal{A}_{{\bf n},m-1}}{\widehat{\sigma}_m w_{{\bf n},m-1}} = \frac{p_{{\bf n},m-1} \widehat {\tau  }_m + b_{{\bf n},m-1}}{w_{{\bf n},m-1}} \in \mathcal{H}(\mathbb{C} \setminus \Delta_m),
\]
where $b_{{\bf n},m-1} = p_{{\bf n},m} + \ell_m p_{{\bf n},m-1}$ and
\[  \frac{\mathcal{A}_{{\bf n},m-1}}{\widehat{\sigma}_m w_{{\bf n},m-1}} = \mathcal{O}\left(\frac{1}{z^{|{\bf n}| -M- \overline{n}_1 - \cdots-\overline{n}_{m-1}-1}}\right) ,\qquad  z \to \infty.
\]
Consequently, $(b_{{\bf n},m-1}/p_{{\bf n},m-1})_{ n \in \Lambda}$ forms a sequence of  incomplete diagonal multi-point Pad\'e approximants of $-\widehat{\tau}_m$ satisfying a)-b) with appropriate values of $n,\kappa$ and $\ell$. Again, $\tau_m$ satisfies Carleman's condition  and  Lemma \ref{BusLop} implies that
\[ h-\lim_{n\in \Lambda} \frac{b_{{\bf n},m-1}}{p_{{\bf n},m-1}} = - \widehat{\tau}_m, \qquad \mbox{inside} \qquad \mathbb{C} \setminus \Delta_m,
\]
which is equivalent to
\[ h-\lim_{n\in \Lambda} \frac{p_{{\bf n},m}}{p_{{\bf n},m-1}} = - \widehat{\sigma}_m^{-1}, \qquad \mbox{inside} \qquad \mathbb{C} \setminus \Delta_m,
\]
We have proved (\ref{convHaus}) for $j=m-1$.

For $j= m-2$, we have shown that there exists a polynomial
$w_{{\bf n},m-2}, \deg w_{n,m-2} = |{\bf n}| -M-\overline{n}_1 - \cdots \overline{n}_{m-3}- 1,$ whose zeros lie on $\Delta_{m-2} $ such that
\[ \frac{\mathcal{A}_{{\bf n},m-2}}{w_{{\bf n},m-2}} =  \frac{p_{{\bf n},m-2} + p_{{\bf n},m-1} \widehat{\sigma}_{m-1} + p_{{\bf n},m} \langle \sigma_{m-1},\sigma_m\widehat{\rangle}}{w_{{\bf n},m-2}} \in \mathcal{H}(\mathbb{C} \setminus \Delta_{m-1})
\]
and
\[  \frac{\mathcal{A}_{{\bf n},m-2}}{w_{{\bf n},m-2}} = \mathcal{O}\left(\frac{1}{z^{|{\bf n}|-M - \overline{n}_1 - \cdots-\overline{n}_{m-2}}}\right), \qquad z \to \infty.
\]
However, using (\ref{s22}) and (\ref{4.4}), we obtain
\[\frac{p_{{\bf n},m-2} + p_{{\bf n},m-1} \widehat{\sigma}_{m-1} + p_{{\bf n},m} \langle \sigma_{m-1},\sigma_m\widehat{\rangle}}{\widehat{\sigma}_{m-1}} =
\]
\[
(\ell_{m-1} p_{{\bf n},m-2}+ p_{{\bf n},m-1} + C_1 p_{{\bf n},m})+ p_{{\bf n},m-2}\widehat{\tau}_{m-1}     - p_{{\bf n},m} \langle  {\tau}_{m-1}, \langle \sigma_{m}, \sigma_{m-1} \rangle \widehat{\rangle},
\]
where $\deg \ell_{m-1} = 1$ and $C_1$ is a constant. Consequently, $  {\mathcal{A}_{{\bf n},m-2}}/(\widehat{\sigma}_{m-1}) $ adopts the form of $\mathcal{A}$ in Lemma \ref{reduc}, $ {\mathcal{A}_{{\bf n},m-2}}/({\widehat{\sigma}_{m-1}w_{{\bf n},m-2} }) \in \mathcal{H}(\mathbb{C} \setminus \Delta_{m-1})$,   and
\begin{equation} \label{Anm} \frac{\mathcal{A}_{{\bf n},m-2}}{\widehat{\sigma}_{m-1}w_{{\bf n},m-2}} = \mathcal{O}\left(\frac{1}{z^{|{\bf n}| -M- \overline{n}_1 - \cdots-\overline{n}_{m-2} -1}}\right), \qquad z \to \infty.
\end{equation}
From  Lemma \ref{reduc} it follows that for $ \nu = 0,\ldots, |{\bf n}| -M- \overline{n}_1 - \cdots-\overline{n}_{m-2} -3$
\[ \int_{\Delta_{m-1}} x^{\nu}  \left({p_{{\bf n},m-2}(x)     - p_{{\bf n},m}(x)   \langle \sigma_{m}, \sigma_{m-1} \widehat{\rangle}(x)} \right) \frac{d  \tau_{m-1}(x)}{w_{{\bf n},m-2}(x)} = 0.
\]

Therefore, $ p_{{\bf n},m-2}- p_{{\bf n},m} \langle \sigma_{m}, \sigma_{m-1} \widehat{\rangle} \in \mathcal{H}(\mathbb{C} \setminus \Delta_m)$ must have at least $|{\bf n}| -D- \overline{n}_1 - \cdots-\overline{n}_{m-2} -2$ sign changes on $\Delta_{m-1}$. This means that there exists a polynomial $w_{{\bf n},m-2}^*, \deg w_{{\bf n},m-2}^* = |{\bf n}| -M- \overline{n}_1 - \cdots-\overline{n}_{m-2} -2,$ whose zeros are simple and lie on $\Delta_{m-1}$ such that
\[ \frac{p_{{\bf n},m-2}- p_{{\bf n},m} \langle \sigma_{m}, \sigma_{m-1} \widehat{\rangle}}{w_{{\bf n},m-2}^*} \in \mathcal{H}(\mathbb{C} \setminus \Delta_m)
\]
and
\[ \frac{p_{{\bf n},m-2}- p_{{\bf n},m} \langle \sigma_{m}, \sigma_{m-1} \widehat{\rangle}}{w_{{\bf n},m-2}^*}  = \mathcal{O}\left(\frac{1}{z^{{|{\bf n}| -M- \overline{n}_1 - \cdots-\overline{n}_{m-3}-2\overline{n}_{m-2}-1}}}\right).
\]
Due to (\ref{cond1}), this implies that $(p_{{\bf n},m-2}/p_{{\bf n},m}), n \in \Lambda,$ is a sequence of incomplete diagonal Pad\'e approximants of $\langle \sigma_{m}, \sigma_{m-1} \widehat{\rangle}$ and this measure satisfies Carleman's condition. Using  Lemma \ref{BusLop} we obtain its convergence in Hausdorff content to $\langle \sigma_{m}, \sigma_{m-1} \widehat{\rangle}$. To prove the other part in (\ref{convHaus}), we divide by $\langle \sigma_m,\sigma_{m-1}\widehat{\rangle}(z)$ use (\ref{s22}) and proceed as we did in the case $j=m$.

Let us  prove (\ref{convHaus}) in general. Fix $j \in \{0,\ldots,m-3\}$ (for $j=m-2,m-1$ it's been proved). Having in mind (\ref{Anj}) we need to reduce $\mathcal{A}_{{\bf n},j}$ so as to eliminate all $p_{{\bf n},k}, k=j+1,\ldots,m-1$. We start out eliminating $p_{{\bf n},j+1}$. Consider the ratio $\mathcal{A}_{{\bf n},j}/\widehat{\sigma}_{j+1}$. Using (\ref{s22}) and (\ref{4.4}) we obtain
\[ \frac{\mathcal{A}_{{\bf n},j}}{\widehat{\sigma}_{j+1}} =    \left(\ell_{j+1} p_{{\bf n},j}+ \sum_{k=j+1}^m \frac{|s_{j+1,k}|}{|\sigma_{j+1}|} p_{{\bf n},j+1} \right) + p_{{\bf n},j}\widehat{\tau}_{j+1} - \sum_{k=j+2}^m p_{{\bf n},k} \langle  {\tau}_{j+1}, \langle s_{j+2,k}, \sigma_{j+1} \rangle \widehat{\rangle},
\]
and ${\mathcal{A}_{{\bf n},j}}/(\widehat{\sigma}_{j+1})$ has the form of $\mathcal{A}$ in Lemma \ref{reduc}, where $ {\mathcal{A}_{{\bf n},j}}/({\widehat{\sigma}_{j+1}w_{{\bf n},j}}) \in \mathcal{H}(\mathbb{C} \setminus \Delta_{j+1})$, and
\[\frac{\mathcal{A}_{{\bf n},j}}{\widehat{\sigma}_{j+1}w_{{\bf n},j}} \in \mathcal{O}\left(\frac{1}{z^{|{\bf n}| -M- \overline{n}_1 - \cdots-\overline{n}_{j} -1}}\right), \qquad z \to \infty.
\]
From  Lemma \ref{reduc}, we obtain that for $\nu = 0,\ldots,|{\bf n}| -M- \overline{n}_1 - \cdots-\overline{n}_{j} -3$
\[ 0 = \int_{\Delta_{j+1}} x^{\nu} \left( p_{{\bf n},j}(x) - \sum_{k=j+2}^m p_{{\bf n},k}  \langle s_{j+2,k}, \sigma_{j+1}  \widehat{\rangle}(x) \right)\frac{d\tau_{j+1}(x)}{w_{{\bf n},j}(x)},
\]
which implies that the function in parenthesis under the integral sign has at least $|{\bf n}|-M - \overline{n}_1 - \cdots-\overline{n}_{j} -2$ sign changes on $\Delta_{j+1}$. In turn, it follows that there exists  a polynomial $\widetilde{w}_{{\bf n},j+1}, \deg \widetilde{w}_{{\bf n},j+1} = |{\bf n}| -M- \overline{n}_1 - \cdots-\overline{n}_{j} -2$, whose zeros are simple and lie on $\Delta_{j+1}$ such that
\[ \frac{p_{{\bf n},j}  - \sum_{k=j+2}^m p_{{\bf n},k}  \langle s_{j+2,k}, \sigma_{j+1}  \widehat{\rangle} }{\widetilde{w}_{{\bf n},j+1}} \in \mathcal{H}(\mathbb{C} \setminus \Delta_{j+2})
\]
and
\[ \frac{p_{{\bf n},j}  - \sum_{k=j+2}^m p_{{\bf n},k}  \langle s_{j+2,k}, \sigma_{j+1}  \widehat{\rangle} }{\widetilde{w}_{{\bf n},j+1}} = \mathcal{O}\left(\frac{1}{z^{|{\bf n}|-M - \overline{n}_1 - \cdots-\overline{n}_{j-1}-2\overline{n}_{j} -1}}\right), \qquad z \to \infty.
\]
Notice that $p_{{\bf n},j+1}$ has been eliminated and that
\[ \langle s_{j+2,k}, \sigma_{j+1}   {\rangle} = \langle \langle \sigma_{j+2},\sigma_{j+1}\rangle, \sigma_{j+3}, \ldots,\sigma_k \rangle, \qquad k = j+3,\ldots,m.
\]

Now we must do away with $p_{{\bf n},j+2}$ in  $p_{{\bf n},j}  - \sum_{k=j+2}^m p_{{\bf n},k}  \langle s_{j+2,k}, \sigma_{j+1}  \widehat{\rangle}$ (in case that $j+2 < m$). To this end, we consider the ratio
\[ \frac{p_{{\bf n},j}  - \sum_{k=j+2}^m p_{{\bf n},k}  \langle s_{j+2,k}, \sigma_{j+1}  \widehat{\rangle} }{\langle \sigma_{j+2},\sigma_{j+1}\widehat{\rangle}}
\]
and repeat the arguments employed above with $\mathcal{A}_{{\bf n},j}$. After $m-j-2$ reductions obtained applying consecutively Lemma \ref{reduc}, we find that there exists a polynomial which we denote $w_{{\bf n},j}^*, \deg w_{{\bf n},j}^* = |{\bf n}| -M- \overline{n}_1 - \cdots-\overline{n}_{j-1}-(m-j-1)\overline{n}_{j} -2$ whose zeros are simple and lie on $\Delta_{m-1}$ such that
\[ \frac{p_{{\bf n},j} - (-1)^{m-j} p_{{\bf n},m} \langle \sigma_{m},\ldots, \sigma_{j+1} \widehat{\rangle}}{w_{{\bf n},j}^*} \in \mathcal{H}(\mathbb{C} \setminus \Delta_m)
\]
and
\[ \frac{p_{{\bf n},j}- (-1)^{m-j} p_{{\bf n},m} \langle \sigma_{m},\ldots, \sigma_{j+1} \widehat{\rangle}}{w_{{\bf n},j}^*}  = \mathcal{O}\left(\frac{1}{z^{{|{\bf n}| - \overline{n}_1 - \cdots-\overline{n}_{j-1}-(m-j)\overline{n}_{j}-1}}}\right), \qquad z\to \infty.
\]
Dividing by $(-1)^{m-j}  \langle \sigma_{m},\ldots, \sigma_{j+1} \widehat{\rangle},$ from here it also follows that
\[ \frac{p_{{\bf n},j}(-1)^{m-j}\langle \sigma_{m},\ldots, \sigma_{j+1} \widehat{\rangle}^{-1} -  p_{{\bf n},m} }{w_{{\bf n},j}^*} \in \mathcal{H}(\mathbb{C} \setminus \Delta_m)
\]
and
\[ \frac{p_{{\bf n},j}(-1)^{m-j}\langle \sigma_{m},\ldots, \sigma_{j+1} \widehat{\rangle}^{-1}- p_{{\bf n},m} }{w_{{\bf n},j}^*}  = \mathcal{O}\left(\frac{1}{z^{{|{\bf n}| - M-\overline{n}_1 - \cdots-\overline{n}_{j-1}-(m-j)\overline{n}_{j}-2}}}\right), \quad z\to \infty.
\]
On account of (\ref{cond1}), these relations imply that $(p_{{\bf n},j}/p_{{\bf n},m}), {\bf n} \in \Lambda,$ is a sequence of incomplete diagonal multi-point Pad\'e approximants of
$(-1)^{m-j}\langle \sigma_{m},\ldots, \sigma_{j+1} \widehat{\rangle}\phantom{ert}$ and $\phantom{e} (p_{{\bf n},m}/p_{{\bf n},j}),  {\bf n} \in \Lambda,$ is a sequence of incomplete diagonal multi-point Pad\'e approximants of $(-1)^{m-j}\langle \sigma_{m},\ldots, \sigma_{j+1} \widehat{\rangle}^{-1}$. Since $\langle \sigma_{m},\ldots, \sigma_{j+1} \widehat{\rangle}^{-1} = \widehat{\tau}_{m,j+1} + \ell_{m,j+1}, \deg \ell_{m,j+1} = 1,$ and the measures $s_{m,j+1}$ and $\tau_{m,j+1}$ satisfy Carleman's condition by Lemma \ref{BusLop} we obtain (\ref{convHaus}).
\end{proof}

\section{ Proof of  Theorem \ref{th1}}

\begin{proof}
The type I  Hermite-Pad\'e  polynomials  $\left(a_{{\bf n},0},\ldots, a_{{\bf n},m}\right)$ with respect to ${\bf f}$ satisfy
\item[i)] $\deg a_{{\bf n},j} \leq n_j -1, j=1,\ldots,m,$
\item[ii)] $a_{{\bf n},0} + \sum_{j= 1}^m a_{{\bf n},j} (\widehat{s}_{1,j}+r_j) = \mathcal{O}(1/z^{|{\bf n}|}) \in \mathcal{H}(\mathbb{C} \setminus \Delta_1)^{\prime}$

Denote by $T:=\prod_{j=1}^{m}t_j$, $D:= \deg T$, multiplying \textit{ii)} by $T$
\begin{equation}\label{imcop}
p_{{\bf n},0} + \sum_{j= 1}^m Ta_{{\bf n},j} \widehat{s}_{1,j} = \mathcal{O}(1/z^{|{\bf n}| -D}) \in \mathcal{H}(\mathbb{C} \setminus \Delta_1)
\end{equation}
where
\begin{equation}\label{p0}
p_{{\bf n},0}=Ta_{{\bf n},0}+\sum_{j= 1}^m Ta_{{\bf n},j}r_j
\end{equation}
Therefore $\left(p_{{\bf n},0},Ta_{{\bf n},1},\ldots, Ta_{{\bf n},m}\right)$ is  an incomplete type I  Hermite-Pad\'e  approximant of the Nikishin system $(s_{1,1},\ldots,s_{1,m})$ with respect to the multi-index $(n_1+D,\ldots, n_m+D)$. From Lemma~\ref{CCTI} it follows that for $j=1,\ldots m$
\begin{equation} \label{convHausj}
h-\lim_{{\bf n}\in \Lambda}\frac{a_{{\bf n}, j}}{a_{{\bf n},m}} = (-1)^{m-j}\widehat{s}_{m,j+1}, \quad h-\lim_{n\in \Lambda}\frac{a_{{\bf n}, m}}{a_{{\bf n},j}} = \frac{(-1)^{m-j}}{\widehat{s}_{m,j+1} },
\end{equation}
and
\begin{equation} \label{convHaus0}
h-\lim_{{\bf n}\in \Lambda}\frac{p_{{\bf n}, 0}}{Ta_{{\bf n},m}} = (-1)^{m}\widehat{s}_{m,1}, \quad h-\lim_{n\in \Lambda}\frac{Ta_{{\bf n}, m}}{p_{{\bf n},0}} = \frac{(-1)^{m}}{\widehat{s}_{m,1} }.
\end{equation}
Substituing (\ref{p0}) and (\ref{convHausj}) into (\ref{convHaus0}) we get
\begin{equation}\label{Ca0}
h-\lim_{{\bf n}\in \Lambda}\frac{a_{{\bf n}, 0}}{a_{{\bf n},m}} = (-1)^{m}\widehat{s}_{m,1}-\sum_{j=1}^{m}(-1)^{m-j}\widehat{s}_{m,j+1}r_j.
\end{equation}

 Due to \textit{ii)}  and (\ref{eq:4}), it follows that
\[ \int x^{\nu}  \mathcal{A}_{{\bf n},1}(x) T(x) d\sigma_1(x) = 0, \qquad \nu = 0,\ldots, |{\bf n}| -  D -2.
\]
where $\mathcal{A}_{{\bf n},1}=a_{{\bf n},1} + \sum_{k= 2}^m a_{{\bf n},k} \widehat{s}_{2,k}$
This implies that $\mathcal{A}_{{\bf n},1}$ has at least $|{\bf n}|  - D -1$ sign changes in $\stackrel{\circ}{\Delta}_1$. \par

Let $N_{\bf n} = \max\{n_1,n_2-1,\ldots,n_m-1\}$ and
 $w_{{\bf n},1}$ be the monic polynomial whose zeros are the points where $\mathcal{A}_{{\bf n},1}$ changes sign in $\stackrel{\circ}{\Delta}_1$. Then
 \begin{equation}\label{ert}
 \frac{\mathcal{A}_{{\bf n},1}(z)}{w_{{\bf n},1}(z)} = \mathcal{O}\left(1/z^{|{\bf n}|   - N_{\bf n}- D}\right) \in \mathcal{H}(\mathbb{C} \setminus \Delta_2).
\end{equation}

Let $\overline{\jmath}$ be the last component of $(n_1,\ldots,n_m)$ such that $n_{\overline{\jmath}} = \min_{j=1,\ldots,m} (n_j)$.  Let us prove that  $a_{{\bf n},\overline{\jmath}}$ has at most  $D $  zeros on $\mathbb{C}\setminus \Delta_{m} $.

From  \cite[Theorem 1.3]{FL4} (see also \cite[Theorem 3.2]{FL4II}), we know that  there exists a permutation $\lambda$ of $(1,\ldots,m)$ which reorders the components of $(n_1,n_2,\ldots,n_m)$ decreasingly, $n_{\lambda(1)} \geq \cdots \geq n_{\lambda(m)},$ and an associated Nikishin system $(r_{2,2},\ldots,r_{2,m}) = {\mathcal{N}}(\rho_{2},\ldots,\rho_m)$ such that
\[  \mathcal{A}_{{\bf n},1} = (q_{{\bf n},1} + \sum_{k=1}^m q_{{\bf n},k} \widehat{r}_{2,k})\widehat{s}_{2,\lambda(1)}, \quad \deg q_{{\bf n},k} \leq n_{\lambda(k)} -1, \quad k=1,\ldots,m,
\]
where  $\widehat{s}_{2,\lambda(1)} \equiv 1$ when $\lambda(1) = 1$. The permutation may be taken so that for all $1 \leq j <k \leq n$ with $n_j = n_k$ then also $\lambda(j) < \lambda(k)$. In this case,  see formulas (31) in the proof of \cite[Lemma 2.3]{FL4}, it follows that $q_{{\bf n},m}$ is either $a_{{\bf n},\overline{\jmath}}$ or $-a_{{\bf n},\overline{\jmath}}$.

Set
\[\mathcal{Q}_{{\bf n},j} := q_{{\bf n},j} + \sum_{k=j+1}^m q_{{\bf n},k} \widehat{r}_{1,k}, \quad j=1,\ldots,m-1, \quad \mathcal{Q}_{{\bf n},m} := q_{{\bf n},m}.\]
From (\ref{ert}), using again (\ref{eq:4}) we get
\[ \int x^{\nu}  \mathcal{Q}_{{\bf n},2}(x)  \frac{d\rho_2(x)}{w_{{\bf n},1}(x)} = 0, \quad \nu = 0,\ldots, |{\bf n}|  - n_{\lambda(1)}- n_{\lambda(2)} - D -2,
\]
which implies that $\mathcal{Q}_{{\bf n},2}$ has at least $|{\bf n}|  - n_{\lambda(1)}- n_{\lambda(2)} - D -1$ sign changes on $\stackrel{\circ}{\Delta}_2$. Repeating the arguments $m-1$ times, it follows that $\mathcal{Q}_{{\bf n},m} = q_{{\bf n},m}$ has at least $n_{\overline{\jmath}} - D -1$ sign changes on $\stackrel{\circ}{\Delta}_m$ which implies that  $a_{{\bf n},\overline{\jmath}}$ has at most $D$ zeros on   $\mathbb{C}\setminus \Delta_m$ because $q_{{\bf n},m} = \pm a_{{\bf n},\overline{\jmath}}$.

The index $\overline{\jmath}$ as defined above may depend on  ${\bf n} \in \Lambda$. Given $\overline{\jmath} \in \{1,\ldots,m\}$, let $\Lambda(\overline{\jmath})$ denote the set of all ${\bf n} \in \Lambda$ such that $\overline{\jmath}$ is the last component of $(n_1,\ldots,n_m)$ satisfying $n_{\overline{\jmath}} = \min_{j=1,\ldots,m} (n_j)$. Fix $\overline{\jmath}$ and suppose that $\Lambda(\overline{\jmath})$ contains infinitely many multi-indices.\par
Should $\overline{\jmath} = m$, then $a_{{\bf n},m}$ has $n_m -D -1$ zeros in $\stackrel{\circ}{\Delta}_m$. According to (\ref{Ca0}) the sequence of rational fractions $\left\{\frac{a_{{\bf n}, 0}}{a_{{\bf n},m}}\right\}, {\bf n}\in \Lambda(m)$, which have at most $D$ poles in $\mathbb{C}\setminus\Delta_m$ converges in Hausdorff content to a function which has exactly $D$ poles in $\mathbb{C}\setminus\Delta_m$ (recall that if $j\neq k$ the
poles of $r_j$ and $r_k$ distinct). Using Gonchar's lemma in \cite{gon} it follows that for all  ${\bf n}\in \Lambda(m)$ with $\left|{\bf n}\right|$ sufficiently large we have that $\deg a_{{\bf n,m}}=n_m-1$ having $a_{{\bf n,m}}$ exactly $n_m-D-1$ zeros on $\Delta_m$ and the rest of its zeros converge to the poles of the $r_j, j=1,\ldots, m$ according to their order. This fact together with (\ref{convHausj}) and Gonchar's lemma again imply (\ref{fund1}) and (\ref{fund0}).\par

Let us prove that for every compact $K \subset \mathbb{C} \setminus \Delta_m$ there exists $N = N(K)$ such that for ${\bf n} \in \Lambda, |{\bf n}| > N,$ the polynomial $a_{{\bf n},m}$ has at most $D$ zeros on $K$. To the contrary, suppose there exists $K \subset \mathbb{C} \setminus \Delta_m$ and an infinite sequence of multi-indices $\Lambda^{\prime} \subset \Lambda$ such that for every ${\bf n} \in \Lambda^{\prime}$ $a_{{\bf n},m}$ has at least $D+1$ zeros on $K$. Since $\cup_{\overline{j} =1}^m \Lambda(\overline{\j}) = \Lambda$ there exists $\overline{\j}$ such that $\Lambda(\overline{\j}) \cap \Lambda^{\prime}$ contains infinitely many sub-indices. Because of what was proved above $\overline{\j} \in \{1,\ldots,m-1\}$.

Fix $R$ sufficiently large so that $K$ is contained in the disk $D(0,R) = \{z:|z| < R\}$. The polynomial $a_{{\bf n},\overline{\j}}, {\bf n} \in \Lambda(\overline{\j}) \cap \Lambda^{\prime},$ has at most $D$ zeros in $D(0,R) \setminus \Delta_m$. Let $q_{\bf n}$ be the monic polynomial of degree $\leq D$ whose zeros are the points in $D(0,R) \setminus \Delta_m$ where $a_{{\bf n},\overline{\j}}$ equals zero. Since the zeros of the polynomials $q_{\bf n}$ are uniformly bounded, there exists an infinite sequence of indices $\widetilde{\Lambda} \subset \Lambda(\overline{\j}) \cap \Lambda^{\prime}$ such that $\lim_{{\bf n} \in \widetilde{\Lambda}}q_{\bf n} = q, \deg q \leq D,$ uniformly of $\overline{D(0,R)}$. Since the number of zeros of $q$ is at most $D$ and the distance between $K$ and $\{z:|z|= R\} \cup \Delta_m$ is positive, we can find a compact set $\widetilde{K} \subset D(0,R) \setminus \Delta_m$ which contains $K$ in its interior, whose boundary $\partial \widetilde{K}$  consists of a finite number of non-intersecting smooth Jordan curves, and $\partial \widetilde{K}$ contains non of the zeros of $q$. Taking a subsequence of $\widetilde{\Lambda}$ if necessary we can assume that $\partial \widetilde{K}$ contains no zero of $q_{\bf n}, {\bf n} \in \widetilde{\Lambda}$.

The functions ${q_{\bf n} a_{{\bf n},m}}/{a_{{\bf n},\overline{\j}}}, {\bf n} \in \widetilde{\Lambda},$ are holomorphic in $D(0,R) \setminus \Delta_m$. On account of the second part in (\ref{convHausj}) and Gonchar's lemma, we obtain that
\[ \lim_{{\bf n} \in \widetilde{\Lambda}} \frac{q_{\bf n} a_{{\bf n},m}}{a_{{\bf n},\overline{\j}}} =  (-1)^{m - \overline{\j}} \frac{q}{\widehat{s}_{m,\overline{\j} +1}}
\]
uniformly on compact subsets of $D(0,R) \setminus \Delta_m$; in particular on $\partial \widetilde{K}$. The limit is never equal to zero on $\partial \widetilde{K}$. Let $\partial \widetilde{K}$ be oriented positively. Then
\[ \lim_{{\bf n} \in \widetilde{\Lambda}}  \frac{1}{2\pi i} \int_{\partial \widetilde{K}} \frac{({q_{\bf n} a_{{\bf n},m}}/{a_{{\bf n},\overline{\j}}})^{\prime}(z)dz}{ ({q_{\bf n} a_{{\bf n},m}}/{a_{{\bf n},\overline{\j}}})(z)} = \frac{1}{2\pi i}\int_{\partial \widetilde{K}}
\frac{({q}/{\widehat{s}_{m,\overline{\j} +1}})^{\prime}(z) dz}{ ({q}/{\widehat{s}_{m,\overline{\j} +1}})(z)}.
\]
According to the argument principle, the right hand equals the number of zeros of $q$ surrounded by $\partial \widetilde{K}$ which is at most $D$. Therefore, for all ${\bf n} \in \widetilde{\Lambda}$ such that $|{\bf n}|$ is sufficiently large the left hand must be equal to an integer $\leq D$. However, for each ${\bf n}$ the integral on the left represents the number of zeros of $a_{{\bf n},m}$ surrounded by $\partial \widetilde{K}$ which is at least equal to $D+1$. This contradiction proofs the statement.

Combining the statement just proved with (\ref{Ca0}) and Gonchar's lemma \eqref{fund0} readily follows. Additionally, Gonchar's lemma implies  we get that each pole
of $r_j, j=1,\ldots,m$ attracts as many zeros of $a_{{\bf n},m}$ as its order. That is, if $\zeta$ is a zero of $T$ of multiplicity $\kappa$ then for each $\varepsilon >0$ sufficiently small there exists an $N$ such that for all ${\bf n} \in \Lambda, |{\bf n}| > N,$  $a_{{\bf n},m}$ has at least $\kappa$ zeros in $\{z:|z-\zeta| < \varepsilon\}$. Since the total number of such zeros counting multiplicities is $D$, we conclude that the only accumulation points of the zeros of the $a_{{\bf n},m}$ are either the poles of the $r_j$ (each of which attracts exactly as many zeros of the $a_{{\bf n},m}$ as its order) or points in $\Delta_m \cup \{\infty\}$. Using again the argument principle this is true for all $j=1,\ldots,m$ and the rest of the zeros of $a_{{\bf n},j}$ acumulate on $\Delta_m \cup \{\infty\}$. This together with  (\ref{convHausj}) and Gonchar's lemma imply  (\ref{fund1}). Finally,  \eqref{fund1} and the argument principle imply that also for each $j=1,\ldots,m-1$ each zero of $T$ attracts exactly as many of $a_{{\bf n},j}$ as its multiplicity and the rest of the zeros of $a_{{\bf n},j}$ accumulate on $\Delta_m \cup \{\infty\}$.
\end{proof}

\begin{remark}
According to formula (17) in \cite[Lemma 2.9]{FL4II}, for each $j=0,\ldots, m-1$
\begin{equation}\label{chile}
 0 \equiv (-1)^{m-j }\widehat{s}_{m,j+1}  + \sum_{k=j+1}^{m-1}(-1)^{m-k}   \widehat{s}_{m,k+1} \widehat{s}_{j+1,k} +
  \widehat{s}_{j+1, m}, \quad z \in \mathbb{C} \setminus (\Delta_{j+1} \cup \Delta_m).
\end{equation}
Combining (\ref{chile}),(\ref{fund1}), and (\ref{fund0}), we obtain
\begin{equation*}
  \lim_{{\bf n} \in \Lambda} \left( \frac{a_{{\bf n},j} + \sum_{k= j+1}^m a_{{\bf n},k} \widehat{s}_{j+1,k}}{a_{{\bf n},m}}\right) = 0, \qquad j=1,\ldots,m-1,
\end{equation*}
and
\begin{equation*}
  \lim_{{\bf n} \in \Lambda} \left( \frac{a_{{\bf n},0} + \sum_{k= 1}^m a_{{\bf n},k} (\widehat{s}_{1,k}+r_k)}{a_{{\bf n},m}} \right)= 0,
\end{equation*}
uniformly on each compact subset $K$ of $(\mathbb{C}\setminus \Delta_{m})^{\prime}$.
\end{remark}

\begin{remark} The thesis of Theorem \ref{th1} remains valid if in place of  \eqref{cond1} we require that
\begin{equation} \label{cond2} n_j  =  \frac{|{\bf n}|}{m}  + o(|{\bf n}|),\qquad   |{\bf n}| \to \infty ,  \qquad j=1,\ldots,m.
\end{equation}
To prove this we need an improved version of Lemma~\ref{CCTI} in which  the parameter $M$ in b')  depends on $n$ but $M(n) = o(n), n \to \infty$.
\end{remark}

\begin{remark} If either $\Delta_m$ or $\Delta_{m-1}$ is a compact set and $\Delta_{m-1} \cap \Delta_{m} = \emptyset$, it not difficult to show that convergence takes place in \eqref{fund1} and \eqref{fund0} with geometric rate. More precisely, for $j=1,\ldots,m-1,$ and $\mathcal{K} \subset \mathbb{C} \setminus \Delta_m$, we have
\begin{equation} \label{asin1} \limsup_{{\bf n} \in \Lambda} \left\|\frac{a_{{\bf n}, j}}{a_{{\bf n},m}} - (-1)^{m-j}\widehat{s}_{m,j+1} \right\|_{\mathcal{K}}^{1/|{\bf n}|} = \delta_j < 1.
\end{equation}
and
\begin{equation} \label{asin2} \limsup_{{\bf n} \in \Lambda} \left\|\frac{a_{{\bf n}, 0}}{a_{{\bf n},m}} -
((-1)^{m}\widehat{s}_{m,1}-\sum_{j=1}^{m-1}(-1)^{m-j}\widehat{s}_{m,j+1}r_j+r_m) \right\|_{\mathcal{K}}^{1/|{\bf n}|} = \delta_0 < 1.
\end{equation}
The proof  is similar to that of \cite[Corollary 1]{FL2}. It is based on the fact that the number of interpolation points on $\Delta_{m-1}$ is $\mathcal{O}(|{\bf n}|), |{\bf n}| \to \infty,$ and that the distance from $\Delta_m$ to $\Delta_{m-1}$ is positive. Relations \eqref{asin1} and \eqref{asin2} are also valid if \eqref{cond1} is replaced with \eqref{cond2}.
Asymptotically, \eqref{cond2} still means that the components of $\bf n$ are equally distributed. One can relax \eqref{cond2} requiring, for example, that the generating measures are regular in the sense of \cite[Chapter 3]{stto} in which case the exact asymptotics of \eqref{asin1} and \eqref{asin2} can be given (see, for example, \cite{Nik0}, \cite[Chapter 5, Section 7]{NS}, and \cite[Theorem 5.1, Corollary 5.3]{FLLS}).
\end{remark}

\begin{remark}  The previous ideas can be applied to other approximation schemes.\par
 Let $(s_{1,1},\ldots,s_{1,m}) = \mathcal{N}(\sigma_1,\ldots,\sigma_m), {\bf n} = (n_1,\ldots,n_m) \in \mathbb{Z}_+^{m} \setminus \{{\bf 0}\},$ and $w_{\bf n},$ $ \deg w_{\bf n} \leq |{\bf n}| + \max(n_j)-2,$ a polynomial with real coefficients whose zeros lie in $\mathbb{C} \setminus\Delta_1$, be given. We say that $\left(a_{{\bf n},0}, \ldots, a_{{\bf n},m}\right)$ is a  type I multi-point Hermite-Pad\'e approximation of $(\widehat{s}_{1,1}+r_1,\ldots,\widehat{s}_{1,m}+r_m)$ with respect to $w_{\bf n}$ if:
\begin{itemize}
\item[i)] $\deg a_{{\bf n},j} \leq n_j -1, j=1,\ldots,m, \quad \deg a_{{\bf n},0} \leq n_0 -1, \quad n_0 := \max_{j=1,\ldots,m} (n_j) -1,\quad$ not all identically equal to $0$ ($n_j = 0$ implies that $a_{{\bf n},j} \equiv 0$),
\item[ii)] $\mathcal{A}_{{\bf n},0}/w_{\bf n} \in \mathcal{H}(\mathbb{C} \setminus \Delta_1)^{\prime}\quad$ and  $\quad \mathcal{A}_{{\bf n},0}(z)/w_{\bf n}(z)= \mathcal{O}(1/z^{|{\bf n}|}), \quad z\to \infty.$
\end{itemize}
Then, a  result analogous to Theorem~\ref{th1} is true.
\end{remark}

\end{document}